\documentclass[12pt,draft]{amsart}
\usepackage{amssymb}
\usepackage{latexsym}
\usepackage{amsfonts}
\usepackage{amsmath}
\usepackage{color}

\newtheorem{theorem}{Theorem}
\newtheorem{proposition}[theorem]{Proposition}
\newtheorem{lemma}[theorem]{Lemma}
\newtheorem{corollary}[theorem]{Corollary}

\theoremstyle{definition}
\newtheorem{remark}[theorem]{Remark}

\newtheorem{example}[theorem]{Example}

\newcommand{\alt}[1]{{\sf A}_{#1}}

\newcommand{\sy}[1]{{\sf S}_{#1}}

\newcommand{\core}[2]{{\sf Core}_{#1}(#2)}

\newcommand{\sym}[1]{{\sf Sym}\,#1}

\renewcommand{\wr}{\,{\sf wr}\,}

\newcommand{\aut}[1]{{\sf Aut}\,{#1}}
\newcommand{\out}[1]{{\sf Out}(#1)}

\newcommand{\rad}[1]{{\sf Rad}{#1}}
\newcommand{\socle}[1]{{\sf Soc}{#1}}

\newcommand{\psl}[2]{\mbox{\sf PSL}_{#1}(#2)}

\newcommand{\F}{\mathbb F}

\renewcommand{\leq}{\leqslant}
\renewcommand{\geq}{\geqslant}

\font\tenvr=cmmi10 scaled 1600
\textfont"F=\tenvr
\renewcommand{\wr}
    {\mathrel{\mkern-1mu\mathchar"0F7B\mkern-1mu}}

\begin{document}

\title[Solvable subgroups of prime-power index]{Groups with a solvable subgroup of prime-power index}
\author{Raimundo Bastos}
\address[R. Bastos]{ Departamento de Matem\'atica, Universidade de Bras\'ilia, Campus Universit\'{a}rio Darcy Ribeiro, Brasilia, DF, 70910-900, Brazil \\ bastos@mat.unb.br}
\author{Csaba Schneider}
\address[C. Schneider]{Departamento de Matem\'atica\\
Instituto de Ci\^encias Exatas\\
Universidade Federal de Minas Gerais\\
Av.\ Ant\^onio Carlos 6627\\
Belo Horizonte, MG, Brazil\\
csaba@mat.ufmg.br\\
 www.mat.ufmg.br/$\sim$csaba}
\subjclass[2010]{20D05; 20D10; 20D20}
\keywords{Finite groups; $\pi$-subgroups; Solvable groups}

\begin{abstract}
In this paper we describe some properties of groups $G$ that contain a solvable subgroup of finite prime-power index (Theorem \ref{thm:solvable} and Corollaries \ref{cor:cor1}--\ref{cor:cor2}).   We prove that if $G$ is a non-solvable group that contains a solvable subgroup of index $p^{\alpha}$ (for some prime $p$), then the quotient $G/\rad(G)$ of $G$ over the solvable radical is asymptotically small in comparison to $p^{\alpha}!$ (Theorem \ref{thm:rad(G)}). 
\end{abstract}

\date{\today}
\subjclass[2010]{}

\thanks{}
\maketitle

\section{Introduction}

In this paper we explore some properties of groups with solvable subgroups of $p$-power index for some prime $p$. 
If the group is finite, then the solvable subgroup of $p$-power index has a Hall $p'$-subgroup which is 
a Hall $p'$-subgroup in the whole group. Thus, in the case of finite groups, our condition is equivalent to requiring that the group contains a solvable 
Hall $p'$-subgroup. There are several well-known results concerning the solvability of finite groups 
assuming the existence of 
certain Hall $p'$-subgroups. The most famous of these is Hall's Theorem that states that a finite 
group is solvable if and only if it contains Hall $p'$-subgroups for all $p$. Furthermore, in a finite solvable group, 
for a fixed $p$, the Hall $p'$-subgroups are conjugate (see \cite[9.1.7--9.1.8]{robinson}). 
Later Wielandt proved that  a finite group is solvable if it possesses three solvable subgroups whose indices are pairwise relatively prime (see for instance \cite[Lemma 11.25]{rose}). In particular, if a finite group $G$ contains solvable Hall $p'$-subgroups for three different primes, then $G$ is solvable. Arad and Ward extended Wielandt's result to show that if $G$ contains a Hall $2'$-subgroup and a Hall $3'$-subgroup, then $G$ is solvable \cite{arad, aw} (see also the work by Herzog, Longobardi, Maj and Mann \cite{hlmm}). Carocca and Matos showed that if $p$ is an odd prime and $G$ is a finite group with a $2$-nilpotent subgroup of $p$-power index, then $G$ is solvable (see \cite[Theorem A]{cm}). More solvability criteria in terms of some (solvable) subgroups can be found in \cite{af,DGHP,HLM,kegel,mt,wielandt}. 

Our first theorem implies, for certain primes $p$, the solvability of finite groups $G$ under the 
condition that $G$ has a solvable Hall $p'$-subgroup. 
Fermat primes are prime numbers of the form $2^m+1$. To simplify notation, 
we set 
\begin{equation}\label{eq:pi}
\pi_0=\{2,7,13\}\cup \{p\mid \mbox{$p$ is a Fermat prime}\}.
\end{equation}

\begin{theorem} \label{thm:solvable}
Let $G$ be a finite group, let $p$ be a prime number and  assume that $G$ contains a solvable  subgroup
of index $p^\alpha$ for some $\alpha\geq 1$. Then $G$ contains a solvable Hall $p'$-subgroup and the following assertions are valid.
\begin{enumerate} 
\item If $p\not\in \pi_0$, then $G$ is solvable.
\item If $p\neq 7,13$, then the Hall $p'$-subgroups of $G$ are conjugate.
\end{enumerate}
\end{theorem}

The proof of Theorem~\ref{thm:solvable} is based on Guralnick's classification \cite{gur} of the finite 
simple groups containing a subgroup of prime-power index. From this classification, it is easy to obtain  
a classification of finite simple groups with a solvable subgroup of prime-power index (see Lemma~\ref{lem:gur}). Guralnick's classification and its consequences for finite simple and characteristically simple groups with a solvable subgroup of prime-power index will be explored in 
Section~\ref{sec:simple}.

Theorem~\ref{thm:solvable} yields the following corollaries.

\begin{corollary}\label{cor:cor1}
If  $p$ is a prime such that $p\not\in\pi_0$ and $G$ is a not necessarily finite group with a solvable subgroup of $p$-power index, then $G$ is solvable.
\end{corollary}

A finite group $G$ is $p$-nilpotent if there exists a normal Hall $p'$-subgroup of $G$. 
It is an immediate consequence of the Feit--Thompson Theorem~\cite{ft} that a  finite $2$-nilpotent group is solvable.  For $p$ odd, a $p$-nilpotent group need not be solvable and part~(3) of the following corollary
gives a sufficient condition for the solvability of $p$-nilpotent groups. We denote by
$\pi(G)$ the set of primes that divide the order of a finite group $G$.

\begin{corollary} \label{cor:cor2}
Let $p$ and $q$ be different primes and let $G$ be a finite group. 
\begin{enumerate}
\item If  $\{p,q\} \neq \{2,7\}$  and 
$G$ contains a solvable Hall $p'$-subgroup and a solvable Hall $q'$-subgroup, then $G$ is solvable.
\item If $G$ contains a solvable Hall $p'$-subgroup with $p\neq 3$ and a Hall $3'$-subgroup, then $G$ is solvable. 
\item If $G$ is $p$-nilpotent,  $q\not\in \pi_0$ and $G$ contains a solvable Hall $\{p,q\}'$-subgroup, then 
$G$ is solvable.

\end{enumerate}
\end{corollary}

Theorem~\ref{thm:solvable} and Corollaries~\ref{cor:cor1}--\ref{cor:cor2} are proved in Section~\ref{sec:struct}.

If $G$ is a group  and $H$ is a solvable subgroup of $G$ with $|G:H|=m=p^{\alpha}$ for some prime $p$, then, noting that ${\sf Core}_G(H)$ is a solvable normal subgroup of $G$ and so is contained in $\rad(G)$ and considering the transitive $G$-action  on the 
coset space modulo $H$, we obtain that 
\[
|G/\rad(G)|\leqslant |G/\mbox{\sf Core}_G(H)|\leqslant m!.
\]
A more careful analysis, in Section~\ref{sec:bound}, of the structure of $G$ gives the following polynomial bound on the 
size of $G/\rad(G)$.

\begin{theorem} \label{thm:rad(G)}
If $p$ is a prime, and $G$ is a (not necessarily finite) group with a solvable subgroup $H$ such that $|G:H|=p^{\alpha}=m$, then $|G/\rad(G)| \leqslant m^5.$ Moreover, if $p\neq 13$, then $|G/\rad(G)|\leqslant m^4$. 
\end{theorem}

\section{Simple groups with solvable subgroups of prime-power index}\label{sec:simple}

If $p=2^k+1$ is a prime number, then $p$ is said to be a Fermat prime. Easy argument shows that if $2^k+1$ is a prime number, then $k=2^m$, and so a Fermat prime must be of the form $2^{2^m}+1$ for some $m\geq 0$. At this moment only five Fermat primes are known, namely $2^1+1=3$, $2^2+1=5$, $2^4+1=17$, $2^8+1=257$, and $2^{16}+1=65537$. The question whether there are more Fermat primes or whether the number of Fermat primes is finite or infinite is open. Catalan's conjecture proved by Mih\u{a}ilescu~\cite{mihai} states that two consecutive natural numbers cannot both be proper powers except for 8 and 9. Thus the following lemma is valid.

\begin{lemma}\label{lem:nt}
If $q$ is a natural number such that $q$ and $q+1$ are both prime-powers, then one of the following holds:
\begin{enumerate}
    \item $q+1$ is a power of $2$ and $q$ is a Mersenne prime;
    \item $q=2^3=8$ and $q+1=3^2=9$;
    \item $q=2^{2^m}$ with some $m\geq 0$ and $q+1$ is a Fermat prime. 
\end{enumerate}
\end{lemma}

The following Lemma is a consequence of Guralnick's classification~\cite{gur} of finite simple groups with a subgroup of prime-power index. Recall the definition of the set $\pi_0$ in equation~\eqref{eq:pi}.

\begin{lemma}\label{lem:gur}
  Suppose that $T$ is a non-abelian simple group and $H< T$ is a
  solvable subgroup of prime-power index $p^\alpha$. Then $p \in \pi_0$ and 
  one of the following holds:
  \begin{enumerate}
  \item $T=\alt 5$, $H = \alt 4$ and $p^\alpha=5$;
  \item $T=\psl 28$, $H$ is the stabiliser of a line in $\F_8^2$, and
    $p^\alpha=9$;
  \item $T=\psl 2q$ with some odd prime $q\geq 5$, $H$ is the stabiliser of a line in $\F_q^2$, and $p^\alpha=q+1$ is a 2-power;
  \item $T=\psl 2{2^{2^m}}$ with some $m\geq 2$, $H$ is the stabiliser of a line in $\F_{2^{2^m}}^2$, $\alpha=1$, and $p$ is a Fermat prime;
  \item $T=\psl 32$, $H$ is the stabiliser of a line or
    a plane in $\F_2^3$, and $p^\alpha=7$;
    \item $T=\psl 33$, $H$ is the stabiliser of a line or
      a plane in $\F_3^3$, and $p^\alpha=13$.
  \end{enumerate}
\end{lemma}
\begin{proof}
  Guralnick~\cite{gur} classified subgroups of prime-power index in finite
  non-abelian simple groups. Our lemma follows by inspection of Guralnick's
  list which contains five cases (a)--(e).
  In Guralnick's case~(a), $H=\alt{n-1}$ which is solvable only for
  $n-1=4$ (that is, $n=5$); this gives item~(1). In case~(b), $T=\psl nq$ and
  $H$ is the stabilizer of either a line or a hyperplane in $\F_q^n$.
  This stabiliser is a parabolic subgroup and it contains a section
  isomorphic to $\psl{n-1}q$. Hence $H$ can only be solvable if
  $\psl {n-1}q$ is solvable which happens only for $n-1=1$ or for $n-1=2$ and
  $q\in\{2,3\}$. If $n-1=1$, then we obtain items (2)--(4). In items (3)--(4), the fact that $q$ and $p$ are primes, respectively, follows
  from Lemma~\ref{lem:nt}. If $n-1=2$ and $q=2,3$, then we obtain items (5)--(6). In cases, (c)--(e) of Guralnick's list, $H$ is non-solvable (in case~(e),
  $\alt 5$ is a composition factor of $H$).
  \end{proof}
  
  \begin{corollary}\label{cor:gur}
  If $T$, $H$, $p$, and $\alpha$  are as in Lemma~\ref{lem:gur}, then the following hold.
  \begin{enumerate}
      \item $H$ is a maximal subgroup of $T$;
      \item in cases (1)--(5), $4\cdot |\aut(T)| \leq p^{4\cdot \alpha}$, while in case (6), $4\cdot|\aut(T)| \leq 13^{5}$;
      \item $H$ is a Hall $p'$-subgroup of $T$;
      \item in cases (1)--(4), $T$ contains a unique conjugacy class of Hall $p'$-subgroups, while in cases (5)--(6), $T$ contains two conjugacy classes of Hall $p'$-subgroups.
  \end{enumerate}
  \end{corollary}

\begin{proof}
(1) The maximality of $H$ follows by noting that $H$ is the point stabiliser of a primitive action of $T$.

(2) In cases, (1), (2), (5) and (6) of Lemma \ref{lem:gur}, the statement follows by 
direct calculation observing that $|\out {A_5}|=|\out{\psl 32}|=|\out{\psl 33}|=2$, while $|\out{\psl 28}|=3$. 
In case~(3), we have $|\out{\psl 2{q}}|=2$ and $q+1=2^\alpha$ for some $\alpha\geq 2$, and so 
\[
4\cdot |\aut(\psl 2q)| \leqslant 4\cdot \dfrac{q \cdot (q^2-1)}{2} \cdot 2 \leqslant 4\cdot (2^{\alpha})^3  \leqslant 2^{4 \cdot \alpha} =   |T:H|^4.
\]
In  case~(4), $p = 2^{2^m}+1$ with some $m\geq 2$, $T = \psl 2{p-1}$, 
and $|\out {\psl 2{p-1}}|=2^m$. Consequently, 
\[
4\cdot|\aut(\psl 2{p-1})|  =  4\cdot (p-1)\cdot p\cdot (p-2) \cdot 2^m\leqslant p^4 = |T:H|^4.
\]

(3) and (4). Let us show that $p\nmid |H|$. If $T=\alt 5$, $\psl 32$, or $\psl 33$, this follows by directly computing the order of $T$ and $H$. If $T=\psl 2q$, then $|T|=q\cdot (q+1)\cdot (q-1)/d$ and $H=q\cdot(q-1)/d$ where $d=\gcd(q+1,2)$. We claim that 
\[\gcd(|H|,|T:H|)=\gcd(q\cdot (q-1)/d,q+1)=1.
\]
If $q$ is even, this is clear, since a prime divisor $r$ of $q+1$ must be odd and so $r$ cannot divide $q$ or $q-1$. If $q$ is odd, then $q+1=p^\alpha$ must be a power of 2. Since $q\geq 3$, $4\mid q+1$, which gives that $4\nmid q-1$. Hence $|H|=q\cdot (q-1)/2$ is odd and this shows that $\gcd(|H|,|T:H|)$ must also be odd. On the other hand, an odd prime $r$ which divides $q+1$ does not divide $q\cdot (q-1)$, and hence $\gcd(|H|,|T:H|)=1$.
  
In particular, in each of these cases, $H$ is a maximal subgroup of $T$ and $H$ is Hall $p'$-subgroup of $T$. Furthermore, in cases (1)-(4), the possible subgroups $H$ are stabilizers of a line in $\F_q^2$ and since $T$ is transitive on the set of such lines, these stabilizers form a single conjugacy class. In cases (5)-(6), $H$ is the stabilizer of either a line or a plane in $\F_q^3$. As $T$ is transitive on the set of lines and also on the set of planes, the stabilizers of lines form a conjugacy class, and another distinct conjugacy class is formed by the stabilizers of the planes. 
\end{proof}

The reason why the constant ``$4$'' appears in Corollary \ref{cor:gur}(2) is that it
also appears in the particular case of  Mar\'oti's result \cite[Corollary 1.5]{maroti} 
stated as~Theorem \ref{thm:maroti} in Section~\ref{sec:bound}.

\begin{lemma}\label{lem:charsimp}
  Suppose that $k\geq 1$, $T_1,\ldots,T_k$ are  finite, non-abelian, simple groups, and set 
  \[
  M=T_1\times\cdots\times T_k.
  \] 
  Suppose that $K$ is a solvable subgroup of 
  $M$ with $p$-power index for some prime $p$. Then $K$ is core-free, 
  and there exists, for each $i\in\{1,\ldots,k\}$, a proper subgroup $H_i< T_i$ 
  such that $|T_i:H_i|=p^{\alpha_i}$, the triple $(T_i,H_i,p^{\alpha_i})$ is as in one of the
  items of Lemma~\ref{lem:gur}, and $K= H_1\times\cdots\times H_k$.
  Furthermore, the following hold.
  \begin{enumerate}
  \item If $p\neq 2$, then the $T_i$ are pairwise isomorphic.
  \item If $p\neq 2$ and $p\neq 3$, then $\alpha_i=1$ for all $i$ and $|M:K|=p^k$; while if $p=3$, then 
  $\alpha_i=2$ for all $i$ and $|M:K|=p^{2k}$.
      \item $K$ is a Hall $p'$-subgroup of $M$.
      \item If $p\neq 7,13$, then $M$ contains a unique conjugacy class of Hall $p'$-subgroups.
  \end{enumerate}
\end{lemma}
\begin{proof}
  Let $\sigma_i:M\rightarrow T_i$ be the $i$-th coordinate projection
  that maps $(t_1,\ldots,t_k)\mapsto t_i$.  For $i\in\{1,\ldots,k\}$, set $H_i=K\sigma_i$. Since $K$ is solvable, $H_i$ is solvable, and in
  particular, $H_i<T$. Since $\core MK$ is a normal subgroup of $M$ and is contained in  $K$, and as the normal subgroups of $M$ are subproducts of some of the $T_i$, we obtain that
  $\core MK=1$ and $K$ is core-free.
 Set $\overline K=\prod_i H_i$.
  Clearly, $K\leq \overline K$. As $|M:K|=p^m$ for some $m$, we have that $|M:\overline K|$
  is a power of $p$, and hence, for all $i$, we have $|T_i:H_i|=p^{\alpha_i}$ for some $\alpha_i$. Thus the triple $(T_i, H_i,p^{\alpha_i})$ is as in one of the items of Lemma~\ref{lem:gur}.
  Since $p\nmid |H_i|$, $p\nmid |\overline K|$, and so $K=\overline K$; that is,
  $K=H_1\times\cdots\times H_k$. 
  Now items~(1)--(2) follow by inspection of the cases in Lemma~\ref{lem:gur}. 
  Let us verify claim~(3). If $L$ is another Hall $p'$-subgroup of $M$, then the same argument shows that $L\cong L_1\times\cdots\times L_k$ where the subgroup $L_i$ of $T_i$ satisfies the same conditions 
  as $H_i$.
  Assuming that $p\neq 7,13$, $T_i$ has a unique conjugacy class of subgroups isomorphic to $H_i$ (Corollary~\ref{cor:gur}), and hence
  $L$ and $K$ are conjugate in $M$.
\end{proof}

\section{The structure of groups with a solvable subgroup of prime-power index}\label{sec:struct}

The next proposition characterizes the possible non-abelian composition factors in groups with a solvable subgroup
of prime-power index.
\begin{proposition}\label{prop:compfact}
Let $G$ be a possibly infinite group with a solvable subgroup $H$ such that $|G:H|=p^\alpha$ with some prime $p$. Suppose that $G/\core GH$ has a non-cyclic composition factor $T$. Then $p \in \pi_0$ and one of the following is valid:
  \begin{enumerate}
  \item $p=2$ and $T\cong \psl 2q$ with some odd prime $q$ such that $q\geq 5$ and $q+1$ is a power of $2$;
  \item $p=3$ and $T \cong \psl 28$;
  \item $p=5$ and $T\cong \alt 5$;
  \item $p=7$ and $T\cong \psl 32$;
  \item $p=13$ and $T\cong \psl 33$;
  \item $p$ is a Fermat prime and $T\cong \psl 2{2^{2^m}}$ with some $m\geq 2$ such that $p=2^{2^m}+1$. 
  \end{enumerate}
\end{proposition}
\begin{proof}
  Suppose without loss of generality that $\core GH=1$. Then $G$ is a finite group, as it can be embedded into $\sy {p^\alpha}$. We proceed by induction on $|G|$. In the  base case of the induction $G$ is  simple and the assertion follows from Lemma~\ref{lem:gur}.  The induction hypothesis is that the assertion of the proposition is valid for groups of order less than $|G|$. Suppose that $G$ is not simple and let $N$ be a minimal normal subgroup of $G$. First we claim that the non-cyclic composition factors of $N$ are as 
  in the relevant item of the proposition.
 If $N$ is abelian, then the claim is trivially true, so suppose that $N$ is non-abelian.  
 As $|N:N\cap H|=|HN:H|$, the index of $N\cap H$ in
  $N$ is a power of $p$. Since $H$ is solvable, $H\neq N$, and in particular $N\cap H$ is a 
  proper subgroup of $N$. By Lemma~\ref{lem:charsimp}, the composition factors of $N$ must be as in the corresponding item of the proposition.
  
  It remains to show that the non-cyclic composition factors of $G/N$ are as in the proposition.
    Since $|G/N:HN/N|=|G:HN|$, we  find that $HN/N$ is a subgroup of $G/N$ with index $p^\beta$ for some $\beta\leq\alpha$. If $\beta=0$, then $HN=G$ and $G/N=HN/N\cong H/(N\cap H)$ which   is solvable. If $\beta\geq 1$, then $G/N$ satisfies the conditions of the induction hypothesis with the subgroup $HN/N$, and so a non-cyclic composition factor of $G/N$ must be as in the corresponding item of the proposition. 
\end{proof}

We are now in a position to prove Theorem \ref{thm:solvable}. 

\begin{proof}[Proof of Theorem \ref{thm:solvable}]
First note that $H$ is a finite solvable group and so it contains a solvable Hall $p'$-subgroup and this subgroup is also a solvable Hall $p'$-subgroup of $G$. If $p$ is a prime not contained in $\pi_0$ (defined in~\eqref{eq:pi}), 
then Proposition~\ref{prop:compfact} implies that all composition factors of $G$ must be cyclic, and so 
$G$ is solvable. This shows part~(1).

(2) 
Let us now suppose that $p\neq 7,13$ and prove that $G$ contains a unique
conjugacy class of Hall $p'$-subgroups. Note that this assertion holds when $G$ is solvable (by Hall's Theorem \cite[9.1.7]{robinson}), or when $G$ is characteristically simple (Lemma~\ref{lem:charsimp}). Our argument goes by induction on $|G|$, the base
case being the case of solvable or characteristically simple groups.
 
Assume that $G$ is not solvable and is not characteristically simple. Let $K_1$ and $K_2$ be two Hall  $p'$-subgroups of $G$ and suppose 
that $K_1$ is solvable. Let $M$ be a minimal normal subgroup of $G$. First, assume that $M$ is an elementary abelian $p$-group.  Note that $K_1M/M$ and $K_2M/M$ are Hall $p'$-subgroups of $G/M$. Thus, by the induction hypothesis, 
$K_1M/M$ and $K_2M/M$ are conjugate in $G/M$; that is $(K_1M/M)^{gM}=K_2M/M$, for some $g\in G$, which implies that $K_1^gM=K_2M$. As $K_1$ and $K_2$ are $p'$-subgroups and $M$ is a $p$-group, $K_1^g$ and $K_2$ are Hall $p'$-subgroups of $Y=K_1^gM=K_2M$. Furthermore, $Y$ is a solvable group, and hence $K_1^g$ and $K_2$ are conjugate in $Y$ by Hall's Theorem (\cite[9.1.7]{robinson}). Therefore $K_1$ and $K_2$
are conjugate in $G$.

Next we suppose that $M$ is an elementary abelian $r$-group with $r\neq p$. Then, 
for $i=1,2$, the product $MK_i$ is a $p'$-subgroup of $G$, but, since the $K_i$ are Hall $p'$-subgroups, we must have $MK_i=K_i$, and so $M\leq K_i$. Furthermore, $K_1/M$ 
and $K_2/M$ are Hall $p'$-subgroups of $G/M$, and hence by the induction hypothesis, $(K_1/M)^{gM}=K_2/M$ for some $g\in G$. That is, $K_1^g=K_2$, and so 
$K_1$ and $K_2$ are conjugate in $G$, as required.

Finally, assume that $M$ is a non-abelian characteristically simple group.  Then
$M\cong T^k$ where $T$ is a finite simple group. Furthermore, $T$ is a non-cyclic
composition factor of $G$, and hence the pair $(p,T)$ is as in one of the items (1)--(3), or~(6) of Proposition~\ref{prop:compfact}. Note that $K_1\cap M$ and $K_2\cap M$ are proper solvable subgroups of $M$ with $p$-power index. By Lemma~\ref{lem:charsimp}, both
$K_1\cap M$ and $K_2\cap M$ are Hall $p'$-subgroups of $M$ and, since $p\neq 7,13$, they are conjugate in $M$.
Hence there exists $n\in N$ such that $(K_1\cap M)^n=K_2\cap M$. Swapping $K_1$ by $K_1^n$, we may assume without loss of generality that $K_1\cap M=K_2\cap M$ and let us call this group $Y$. If $K_1=K_2$, then there is nothing more to prove, and so suppose that $K_1\neq K_2$. Set $N=N_G(Y)$. Since $1<Y<M$, the subgroup $N$ is proper in $G$, such that $K_1,K_2\leq N$. Furthermore, $K_1$ and $K_2$ are Hall $p'$-subgroups of $N$ with $K_1$ being solvable. By the induction hypothesis $K_1$ and $K_2$ are conjugate in $N$, and in particular, they are conjugate in $G$. \end{proof}

\begin{remark}
The examples in Lemma~\ref{lem:gur} show that the condition imposed on the prime $p$ in both 
assertions of Theorem~\ref{thm:solvable} is necessary.
Moreover, Theorem \ref{thm:solvable}(1) improves  the solvability criterion given by Carocca and Matos
\cite[Theorem A]{cm}
for all $p \not\in \pi_0$, because we only ask for the solvability of $H$ rather than its $2$-nilpotency.
\end{remark}

Let us now turn to the proof of Corollaries~\ref{cor:cor1}--\ref{cor:cor2}.

\begin{proof}[The proof of Corollary~\ref{cor:cor1}]
Since $H$ is a solvable subgroup of $G$ with prime-power index, $G/\core GH$ is a finite group with a solvable subgroup $H/\core GH$. Furthermore, $|G/\core GH:H/\core GH| = |G:H| =p^{\alpha}$. By Proposition~\ref{prop:compfact}, the quotient  $G/\core GH$ is solvable. However, since $\core GH$ is contained in $H$, it is solvable, and so the whole group $G$ is solvable. 
\end{proof}

\begin{proof}[The proof of Corollary~\ref{cor:cor2}] 
(1) If $G$ contains a solvable Hall $p'$-subgroup and a solvable Hall $q'$-subgroup for two distinct 
primes $p$ and $q$, then a 
possible non-cyclic composition factor of $G$ would appear in two distinct lines of Proposition~\ref{prop:compfact}.
The only isomorphism type that appears in two distinct lines of Proposition~\ref{prop:compfact} is  $T = \psl 27 \cong \psl 32$, but these two lines correspond to $p=2$ and $p=7$, respectively, and 
this case was excluded. Thus such a group $G$ must be solvable.

(2) Let $H$ and $K$ be a solvable Hall $p'$-subgroup and a Hall $3'$-subgroup of $G$, respectively. 
If $p\not\in \pi_0$, then Theorem \ref{thm:solvable}(1) implies that $G$ is solvable. 
So we may suppose that $p\in\pi_0$. Since $H \cap K$ is a solvable Hall $\{3,p\}'$-subgroup of $G$, 
$H \cap K$ is a solvable Hall $p'$-subgroup of $K$. Furthermore, $3$ does not divide $|K|$. If $K$ were non-solvable, then  a non-cyclic composition factor $T$ of $K$ would be isomorphic to the group in one of the lines of Proposition~\ref{prop:compfact}. Now, inspection of the orders of the possible groups $T$ shows in all cases  
that $3\mid |T|$, which is a contradiction. Consequently, $K$ is solvable. Finally, the solvability of $G$ follows from item (1).

(3) By assumption, $G = HP$, where $H$ is a normal Hall $p'$-subgroup of $G$ and $P$ is a Sylow $p$-subgroup of $G$. Since $H$ is a normal complement of $P$, it is sufficient to show that $H$ is solvable. 
Let $K$ be a solvable Hall $\{p,q\}'$-subgroup of $G$. Since $H$ is a normal $p'$-subgroup, $HK$ is a $p'$-subgroup of $G$. Further, since $H$ is a Hall $p'$-subgroup, $K \leqslant H$, and so $K$ is a solvable Hall $q'$-subgroup of $H$. As $q \not\in \pi_0$, Theorem \ref{thm:solvable}(1) gives that $H$ is solvable, and so $G=HP$ must also be solvable. 
\end{proof}

\begin{remark}
\begin{enumerate}
    \item Corollary~\ref{cor:cor2}(1) is no longer valid if the assumption of the solvability of  both Hall $p'$-subgroups is dropped. For instance, if $q$ is a prime, $q \geqslant 7$, then the group $G=\alt 5 \times C_q$ contains a solvable Hall $5'$-subgroup $H_1 = \alt 4 \times C_q$ and a non-solvable Hall $q'$-subgroup $H_2 = \alt 5$, and  $G$ is non-solvable.  
    \item When $p=2$, Corollary~\ref{cor:cor2}(2) coincides with solvability criterion given by Arad and Ward \cite[Corollary 4.4]{aw}.
    \item In Corollary~\ref{cor:cor2}(3), the normality of the $p$-complement is essential. For instance, consider $G=\psl{2}{11}$, $p=5$ and $q=11\not\in\pi_0$. The group $G$ is non-solvable  with an $11$-complement and a solvable Hall $\{5,11\}'$-subgroup.
\end{enumerate}
\end{remark}

\section{The bound on the order of $G/\rad(G)$}\label{sec:bound}

A group $G$ such that $G/\rad(G)$ is finite has a characteristic series (often referred to as the {\em radical series} and is used in computational group theory; see~\cite[10.1.1]{heo} or \cite[Section 3.3]{bhlgo}) of subgroups 
\begin{equation}\label{eq:radser}
1 \leqslant \rad(G) \leqslant A \leqslant B \leqslant G
\end{equation}
where the terms of the series are defined as follows. 
\begin{enumerate}
    \item $\rad(G)$ is the solvable radical of $G$. 
    \item The subgroup $A$ is defined to be the complete inverse image in $G$ of the socle (that is, 
    the product of the  minimal normal subgroups) $\socle(G/\rad(G))$ of $G/\rad(G)$. 
    \item Thus $A/\rad(G)$ is the direct product of a uniquely defined set $\Delta$ of non-abelian simple groups. This set is permuted by $G$, by conjugation, and $B$ is the kernel of the $G$-action on $\Delta$.
    \item Since the centraliser of $A/\rad(G)$ in $G/\rad(G)$ is trivial, the factor $B/A$ is isomorphic to a subgroup of the direct product of the outer automorphism groups of the simple factors of $A/\rad(G)$, and is therefore solvable. The factor $G/B$ can be regarded as a permutation group on the (generally small) set $\Delta$.
\end{enumerate}

The following theorem follows from taking $d=5$ in~\cite[Corollary~1.5]{maroti}.

\begin{theorem}\label{thm:maroti}
If $X$ is a subgroup of $S_k$ such that $X$ has no composition factor isomorphic to $A_n$ with 
$n\geq 6$, then $|X|\leq 4^{k-1}$.
\end{theorem}

Let us now prove Theorem~\ref{thm:rad(G)}. 

\begin{proof}[Proof of Theorem \ref{thm:rad(G)}]
We may assume without loss of generality that $G$ is non-solvable. Then, Corollary~\ref{cor:cor1} implies that $p \in \pi_0$. 

Since $H$ is a solvable subgroup of prime-power index, $G/\core GH$ is a finite group in which the subgroup $H/\core GH$ is solvable of index $|G/\core GH:H/\core GH| =p^{\alpha}$. Set $m = p^{\alpha}$. In particular, $\core GH\leqslant \rad(G)$ and $G/\rad(G)$ is a finite group with a solvable subgroup of index at most $m = p^{\alpha}$. Now, passing to the quotient $G/\rad(G)$, we can assume without loss of generality that $\rad(G)=1$ and it is sufficient to study the order $|G|$.    

Consider the radical series of $G$ defined in~\eqref{eq:radser}. Since $\rad(G)=1$, it follows that $A$ is a product of the  minimal normal subgroups of $G$ and these are all non-abelian. Set $\Delta = \{T_1,\ldots,T_k\}$ to be the 
simple factors of $A$.  Then $A=T_1\times\cdots\times T_k$ and $H\cap A$ is a solvable subgroup of 
$A$ of index $p^\beta$ with $1\leq \beta\leq \alpha$. By Lemma~\ref{lem:charsimp}, 
each $T_i$ contains a proper solvable subgroup $H_i$ of index $p^{\beta_i}$ such that $A\cap H=H_1\times\cdots\times H_k$ and $|A:A\cap H|=p^{\beta_1+\cdots+\beta_k}=p^\beta\leq p^\alpha$. 
Since $\beta_i\geq 1$ for all $i$, it also follows that $k\leq \beta\leq \alpha$. Moreover,
Corollary~\ref{cor:gur} implies 
the inequality $4 \cdot |\aut(T_i)|\leq p^{5\beta_i}$, 
and in fact $4 \cdot |\aut(T_i)|\leq p^{4\beta_i}$ when $p\neq 13$.
Furthermore, the quotient $B/A$ is isomorphic to a subgorup of  $ {\sf Out}(T_1) \times \cdots \times 
{\sf Out}(T_k)$ and $X=G/B$ is a subgroup of $\sym(\Delta)$. 

By Proposition~\ref{prop:compfact}, $G$ has no composition factor isomorphic to 
$A_n$ with $n\geq 6$, and so $|X|\leq 4^{k-1}$ (Theorem~\ref{thm:maroti}).

Bounding the individual quotients
of the radical series~\eqref{eq:radser} and using the estimate in Corollary~\ref{cor:gur}(2), we obtain that 
\begin{align*}
\left|G\right| &\leqslant |X| \cdot  \left( \displaystyle \prod_{i=1}
^{k}  |\aut(T_i)|\right) \leqslant   \left( \displaystyle \prod_{i=1}
^{k} 4 \cdot |\aut(T_i)|\right) \\ & \leqslant  \left( \displaystyle \prod_{i=1}
^{k}  p^{5\beta_i}\right) \leqslant p^{5\cdot (\beta_1+ \cdots + \beta_k)} \leq 
 p^{5\beta}\leq p^{5\alpha} =  m^5.
\end{align*}
If $p\neq 13$, then Corollary~\ref{cor:gur}(2) implies that the constant ``5'' can be replaced by ``4''.
\end{proof}

\section{Examples}\label{sec:ex}

 In this final section, we present two constructions that produce complex examples 
 of groups that have solvable subgroups of prime power index.

\begin{example}\label{exe:01}
Suppose that $G$ is a group and $H$ is a solvable subgroup of $G$ with $|G:H|=p^\alpha$ for some prime $p$ and for $\alpha\geq 1$. Suppose that $K$ is a solvable subgroup of the symmetric group $\sy\ell$ of degree $\ell$. Then the subgroup $H\wr K$ of the wreath product $W=G\wr K$ is solvable and its index is $p^{\ell\alpha}$. Note, in this construction,  that $K$ is not assumed to be transitive. In particular, if $K=1$, then $W=G^\ell$ and its subgroup of $p$-power index is $H^\ell$. 
\end{example}

\begin{example}\label{exe:02}
Let $G$ be a group, let $H\leq G$ be a solvable subgroup with $|G:H|=p^\alpha$ for some prime $p$ and for $\alpha\geq 1$. Suppose that $K$ is a solvable group and $L\leq K$ such that $|K:L|=p^\beta$ with $\beta\geq 1$. Consider $G$ as a group acting on the right coset space $[G:H]$ and define 
the wreath product $W=K\wr G$ with respect to this action; hence $W=K^{p^\alpha}\rtimes G$. Then the subgroup $(L\times K^{p^\alpha-1})\rtimes H$ is a solvable subgroup of $W$ with index $p^{\alpha+\beta}$
\end{example}

\section*{Acknowledgment}
The first author was supported by FAPDF-Brazil. The second author acknowedges the 
support of the CNPq projects {\em Produtividade em Pesquisa} (project no.: 308212/2019-3)  
and {\em Universal} (project no.: 421624/2018-3). 

\bibliographystyle{alpha}
\bibliography{paper}

\end{document}